\documentclass{ifacconf}

\makeatletter
\let\IFACoriginalssect\@ssect
\makeatother

\usepackage{natbib}

\usepackage[
    colorlinks=true,
    linkcolor=blue,
    citecolor=blue,
    urlcolor=blue,
    bookmarks=false
]{hyperref}

\makeatletter
\let\@ssect\IFACoriginalssect
\makeatother

\usepackage{anyfontsize}
\usepackage{graphicx}      
\usepackage{mathtools}

\usepackage{amssymb, amsmath, bm}
\usepackage{amsfonts}
\usepackage{xcolor}
\usepackage{balance}

\usepackage[labelformat=simple]{subcaption}

\usepackage{enumitem}
\newcommand\numberthis{\addtocounter{equation}{1}\tag{\theequation}}

\usepackage{chngcntr}
\makeatletter 
\newcounter{parentassumption} 

\makeatother 

\makeatletter
\@ifundefined{proof}{%
    \newenvironment{proof}
    {\par\noindent\textit{Proof.}\ }
    {\hfill$\square$\par}
}{}
\makeatother

\allowdisplaybreaks

\newtheorem{definition}{Definition}
\newtheorem{assumption}{Assumption}
\newtheorem{lemma}{Lemma}
\newtheorem{remark}{Remark}
\newtheorem{theorem}{Theorem}
\newtheorem{corollary}{Corollary}

\begin{document}
\begin{frontmatter}

\title{Defense Filters for Robust Learning of Partially Observed Dynamical Systems\thanksref{footnoteinfo}} 

\thanks[footnoteinfo]{This work was supported by the U. S. Army Research Laboratory and the U. S. Army Research Office under Grant W911NF2010219, Office of Naval Research under Grant N000142412673, and NSF.}

\author[First]{Jihun Kim} 
\author[First]{Javad Lavaei} 

\address[First]{Department of Industrial Engineering and Operations Research, \\University of California, Berkeley, CA 94720 USA \\(e-mail: \{jihun.kim, lavaei\}@berkeley.edu).}

\begin{abstract}                
Robust estimation of unknown system dynamics from partial observations under attack is a core challenge in learning for control, increasingly critical in multi-agent networks subject to adversarial data poisoning. Existing methods either require highly restrictive assumptions on the adversary or yield weak error bounds. In this paper, we consider identifying partially observed linear dynamical systems subject to adversarial attacks with an attack probability $p < 0.5$ at each time. We propose a defense-filtering framework leveraging knowledge of the observation matrix to construct a defense filter satisfying full column rank and subspace separability conditions, which enables the evaluation of a single-attack scenario against observed data. Using this framework, we analyze the constrained $\ell_2$-norm estimator and establish finite-sample error bounds for recovering the Markov parameter matrix up to a linear transformation. This is the first approach to provide provable system identification guarantees via a defense filter for partially observed systems subject to harsh environments where the attack probability approaches $0.5$.
\end{abstract}


\end{frontmatter}
\section{Introduction}
Learning the dynamics of partially observed systems under adversarial attacks presents a fundamental challenge in cyber-physical security \citep{pasqualetti2013cyber, fawzi2014secure, showkatbakhsh2016identification} and the security of multi-agent networks \citep{chen2019resilient, zhang2021datadriven}. When an adversary injects arbitrary attack signals into system dynamics without structural restrictions, identifying the underlying state-space parameters becomes inherently ill-posed. Existing studies have thus considered adversaries with restricted access to actuators \citep{xie2017secure, showkatbakhsh2020securing}. Such restrictions are motivated by practical security architectures in which certain system components are physically protected or inaccessible to the adversary,  also illustrated in power-grid security  \citep{liu2009false}.
Motivated by these considerations, we investigate a \textit{defense filter} architecture that restricts adversarial injections to a lower-dimensional subspace. In particular, adversarial attacks can enter through a pre-determined set of channels represented by a defense filter. This formulation captures the possibility of selectively securing system components while preserving the interfaces necessary for normal operation, which is consistent with resource-constrained security design \citep{teixeira2015resource, pasqualetti2015control}.

We consider estimating the unknown Markov parameter matrix  from partially observed linear  systems. 
Prior works \citep{sarkar2021finite, oymak2022revisit} have shown that the least-squares method achieves an $\mathcal{O}(1/\sqrt{T})$ estimation error for an observation trajectory of length $T$ subject to independent,  zero-mean Gaussian noise. 

Beyond such benign noise, our recent works \citep{kim2025bridging, kim2025system} introduced estimation frameworks for attack-corrupted trajectories using an unconstrained $\ell_2$-norm or $\ell_1$-norm estimator.
While those techniques provide mathematically rigorous finite-sample guarantees, attack probabilities need to diminish for the resulting error bounds to decay. Extending these results to non-vanishing attack rates is crucial for practicality in persistent adversarial environments.

In this paper, we address this gap by analyzing the constrained $\ell_2$-norm estimator under any fixed attack probability $p < 0.5$.  Our contributions are as follows:
\begin{itemize}[leftmargin=0.4cm]
    \item We formulate geometric conditions on the defense filter:  full column rank and subspace separability that allow a single-attack scenario to be tested against observations.
    \item We establish finite-sample error bounds for recovering the Markov parameter matrix up to a linear transformation under adversarial attacks. We provide an estimation bound for every $p<0.5$ and show that accurate estimation is  achieved   whenever $p<\min\{0.5, 1-\rho\}$, where $\rho\in(0, 1]$ denotes a contraction factor of the true system. 
\end{itemize}

\textbf{Notation.} Given functions $u$ and $v$, $u(x) = \mathcal{O}(v(x))$ implies $u(x) \le c_1 v(x)$ for some $c_1>0$, and $u(x) = \Omega(v(x))$ implies $u(x) \ge c_2 v(x)$ for some $c_2>0$. The notations $\tilde{\mathcal{O}}$ and $\tilde{\Omega}$ hide logarithmic factors. Let $\mathcal{N}(\mu,\Sigma)$ denote the Gaussian distribution with mean $\mu$ and covariance $\Sigma$. For a matrix $A$, let $\|A\|_F$ denote its Frobenius norm and $\|A\|_2$ denote its spectral norm. $I_n$ denotes the $n \times n$ identity matrix  and $\operatorname{Im}(\cdot)$ denotes the column space of a matrix.

\textbf{Outline.} The remainder of the paper is organized as follows. Section~\ref{sec:probform} formulates the problem and Section~\ref{sec:analysis} presents the main theorem and key supporting lemmas. Section~\ref{sec:exp} validates our theoretical analysis through numerical experiments, and Section~\ref{sec:con} offers concluding remarks.

\section{Problem Formulation}\label{sec:probform}

Consider a partially observed discrete-time system represented as
\begin{equation}
\begin{split}\label{sysd}
    x_{t+1} &= A x_t + Bu_t+Ew_t,\\y_t&= Cx_t + Du_t, \hspace{10mm}t=0,1,\dots,
    \end{split}
\end{equation}
where $x_t \in \mathbb{R}^n$ is the state and $y_t \in \mathbb{R}^r$ is the observation output at time $t$. The system matrices $A \in \mathbb{R}^{n \times n}$, $B \in \mathbb{R}^{n \times m}$, and $D \in \mathbb{R}^{r \times m}$ are unknown.
 $u_t \in \mathbb{R}^m$ is the control input that can be freely designed at each time $t$ to accelerate the learning of the unknown system matrices. 
 $w_t \in \mathbb{R}^e$ is the adversarial attack injected into the state update. At each time step $t$, an attack time occurs with probability $p < 0.5$,  and 
$w_t$ is identically zero when the system is not under attack. We allow $w_t$ to be adversarially chosen by an attacker at the attack times. 

The observation matrix $C \in \mathbb{R}^{r \times n}$ is assumed to be known. Based on $C$, the system designer constructs a defense filter matrix $E \in \mathbb{R}^{n \times e}$ to ensure that the resulting filtered attack $Ew_t\in\mathbb{R}^n$ is restricted to the subspace $\operatorname{Im}(E)$. We impose the full column-rank condition on $E$:
\begin{assumption}[Full column-rank]\label{as:fullcolumn1}
   Suppose that
   \begin{align*}
       \operatorname{rank}(CE) = e.
   \end{align*}
\end{assumption}
 A necessary condition for Assumption~\ref{as:fullcolumn1} is $r \ge e$, meaning the attack dimension cannot exceed the  output dimension.

In this paper, our objective is to accurately recover the underlying system dynamics $(A, B, D)$ from the input-output trajectory $\{u_t,y_t\}_{t\ge 0}$. To establish the input-output relationship, we recall the Markov parameter matrix:
\begin{definition}[Markov parameter matrix]\label{markov}
  From the true system $(A,B,C,D)$, the Markov parameter matrix of order $k$ required to recover the input-output mapping is denoted as $G^*$ and defined by 
    \begin{equation}\label{gstar}
        G^* = [D~CB~ CAB~ \dots ~ CA^{k-2}B].
    \end{equation}
    Note that the size of $G^*$ depends on $k$; however, we omit the explicit dependence on $k$ in the notation for simplicity.
\end{definition}

An accurate estimate of the Markov parameter matrix allows for the recovery of the system matrices $A, B,$ and $D$ up to a similarity transformation. In particular, subspace identification techniques such as the Ho-Kalman algorithm (\cite{hokalman1966kalman}) have been shown to yield accurate $k$-th order balanced realizations from finite-sample estimates (\cite{sarkar2021finite, oymak2022revisit}) once $G^*$ is sufficiently accurate. Consequently, our objective reduces to obtaining a reliable estimate of $G^*$.

To construct an estimate of $G^*$, we  solve the constrained optimization problem:
\begin{equation}\label{l2est}
    \min_{\substack{H \in \mathbb{R}^{e \times r} \\ J \in \mathbb{R}^{e \times mk}}} \sum_{t\in\mathcal{T}} \bigr\| H y_t - J \mathbf{U}_t^{(k)} \bigr\|_2 
    \quad \text{s.t.} \quad HCE = I_e,
\end{equation}
where
$\mathcal{T}=\{k-1,\dots, T+k-2\}$ and
$\mathbf{U}_t^{(k)} = \begin{bmatrix} u_t^T & u_{t-1}^T & \dots & u_{t-k+1}^T \end{bmatrix}^T \in \mathbb{R}^{mk}$ denotes the stacked control inputs over time $t,t-1,\dots, t-k+1$.

Here, $H \in \mathbb{R}^{e \times r}$ acts as a transformation matrix designed to capture the immediate prior attack $w_{t-1}$ from $y_t$ via the  constraint $HCE = I_e$, while $J \in \mathbb{R}^{e \times mk}$ serves as an auxiliary matrix approximating the transformed matrix  $HG^*$. We learn this matrix pair from a single  observation trajectory, obtaining estimates $(\hat H, \hat J)$. The resulting estimation residual $\hat Hy_t -\hat J \mathbf{U}_t^{(k)}$ then serves as an attack indicator; we classify a time step $t$ as clean if the  residual stays small for consecutive steps, ensuring historical attack effects have  been sufficiently cleared. With this clean dataset identified, we finally estimate $G^*$.

The following two assumptions ensure that system states do not explode over time.
\begin{assumption}[System Stability]\label{stability}
There exist constants $\psi > 0$ and $\rho \in [0, 1)$ such that $\|A^t\|_2 \le \psi \cdot \rho^t$ for all $t \ge 0$.
\end{assumption}

\begin{assumption}[Sub-Gaussian norm]\label{maxnorm}
 Define a filtration {$\mathcal{F}_t = \bm{\sigma}\{x_0, x_1, \dots, x_t\}$.}
    There exists $\eta >0$ such that
  {   $x_0 = (x_{0,i})_{i=1}^n$ and $w_t = (w_{t,i})_{i=1}^e$ satisfies $\max_{i\in\{1,\dots,n\}}\|x_{0,i}\|_{\psi_2} \leq \eta$ and $\max_{i\in\{1,\dots,e\}}\|w_{t,i}\;|\;\mathcal{F}_t\|_{\psi_2} \leq \eta$ for all $t \geq 0$.}
\end{assumption}

We define the probabilistic attack model using independent Bernoulli variables, under the mild condition that the attack probability at each step is strictly less than one half.
\begin{assumption}[Probabilistic Attack]\label{null}
 $w_t$ is an attack at each time $t$ with probability $p<0.5$ conditioned on $\mathcal{F}_t$; \textit{i.e.}, there exists a sequence $(\xi_t)_{t\geq 0}$ of independent $\text{Bernoulli}(p)$ variables such that 
    $\{\xi_t = 0\} \subseteq \{w_t=0\}$ for all $t\geq 0.$
   Under this restriction on attack times, $w_t$ can be chosen arbitrarily by an adversary with access to $\mathcal{F}_t$ at every attack time $t\ge 0$.
\end{assumption}

By designing the control inputs $u_0, u_1, \dots$ and given the associated observation trajectory $y_0, y_1,\dots$,
our goal is to accurately approximate the input–output mapping $G^*$; to identify a Markov parameter matrix of order $k$. To this end, {for each $s\in\{1,\dots,k-1\}$,
consider the following relationship from $u_t,\dots, u_{t-k+1}$ to $y_t$ derived from \eqref{sysd}:
    \begin{subequations}\label{relationship}
         \begin{align}
    y_t &= G^*\cdot [u_t^T~ u_{t-1}^T ~ \dots~ u_{t-k+1}^T]^T \numberthis\label{firstline}\\ &\hspace{2mm}+ [CE~ CAE ~ \dots ~ CA^{s-1}E] \cdot[w_{t-1}^T~ \dots ~ w_{t-s}^T]^T\numberthis\label{secondline}\\&\hspace{2mm}+ [CA^{s}E~  \dots ~ CA^{k-2}E] \cdot[w_{t-s-1}^T~ \dots ~ w_{t-k+1}^T]^T\numberthis\label{thirdline}\\&\hspace{2mm}+ CA^{k-1}x_{t-k+1}, \numberthis\label{fourthline}
    \end{align}
    \end{subequations}
    where the term in \eqref{thirdline} is taken to be zero when $s = k-1$.
For the sake of convenience, we define \eqref{secondline}-\eqref{fourthline} as
\begin{subequations}
\begin{align}
    &v_t^{(s)} := [CE~ CAE ~ \dots ~ CA^{s-1}E] [w_{t-1}^T~ \dots ~ w_{t-s}^T]^T\numberthis\label{defv}\\
    &\alpha_t^{(s)} := [CA^{s}E~  \dots ~ CA^{k-2}E] [w_{t-s-1}^T~ \dots ~ w_{t-k+1}^T]^T \numberthis\label{defalpha}\\
   &\beta_t:=  CA^{k-1}x_{t-k+1} \numberthis\label{defbeta}
\end{align}
\end{subequations}
}

The next lemma bounds the tail terms $\alpha_t^{(s)}$ and $\beta_t$.
\begin{lemma}\label{alphabetabound}
    Suppose that Assumptions \ref{stability} and \ref{maxnorm} hold.  Let the control inputs $\{u_t\}_{t=0}^{T+k-2}$ be chosen independent with $u_t\sim N(0, \sigma^2 I_m)$. Then,  given $\delta\in(0,1]$, when $T= \Omega(\log(\frac{1}{\delta }))$,
    \begin{align*}
    \sum_{t\in \mathcal{T}} \|\alpha_t^{(s)}\|_2 +\|\beta_t\|_2 \le  \mathcal{O}\bigr({\rho^{s}}\|C\|_2 MT\bigr)
\end{align*}
holds with probability at least $1-\delta$, where $M:= \frac{\eta {\sqrt{n}}+ \eta\sqrt{e}\|E\|_2+\sigma \sqrt{m} \|B\|_2}{1-\rho}$.
\end{lemma}

We aim to demonstrate that estimation accuracy improves with larger values of $s$, while preserving the mild assumption on the attack probability $p<0.5$. To establish error bounds for a given $s$, we introduce a geometric subspace condition that strengthens Assumption~\ref{as:fullcolumn1}.
\begin{assumption}[Subspace separability]\label{as:fullcolumn}
  Suppose that 
     \begin{align*}
         \operatorname{Im}(CE)\cap
\operatorname{Im}([CAE\ \dots\ CA^{s-1}E])
=\{0\}.
     \end{align*} 
\end{assumption}
Assumption \ref{as:fullcolumn1} implies that $CE$ preserves all $e$ independent directions, while Assumption \ref{as:fullcolumn} implies that none of these directions can be represented as a linear combination of the directions generated by $CAE,\dots,CA^{s-1}E$.
These assumptions provide a following equivalent condition.

\begin{lemma}\label{assumptionequivalence}
There exists a matrix $H \in \mathbb{R}^{e \times r}$ such that
\[
H CE = I_e \quad \text{and} \quad H \cdot  [CAE ~ \dots ~ CA^{s-1}E] = 0
\]
if and only if
Assumptions \ref{as:fullcolumn1} and \ref{as:fullcolumn} concurrently hold.
\end{lemma}

Substituting the decomposition \eqref{relationship} into the constrained $\ell_2$-norm estimator \eqref{l2est} implies 
\begin{equation*}
\begin{split}
 \min_{\substack{H \in \mathbb{R}^{e \times r} \\ J \in \mathbb{R}^{e \times mk}}} &\sum_{t\in\mathcal{T}} \big\| (H G^* - J) \mathbf{U}_t^{(k)} +Hv_t^{(s)} + H\alpha_t^{(s)} + H\beta_t\big\|_2 
    \\ \text{s.t.} &\quad HCE =I_e.
\end{split}
\end{equation*}

Define $K:= HG^*-J$, $Q_s:=H\cdot[CE~ CAE ~ \dots ~ CA^{s-1}E]$, and $\mathbf{W}_t^{(s)}:=[w_{t-1}^T~ \dots ~ w_{t-s}^T]^T$, and rewrite the optimization problem as  
\begin{align*}
 \min_{\substack{K \in \mathbb{R}^{e \times mk}, Q_s \in \mathbb{R}^{e \times es}  \\ H \in \mathbb{R}^{e \times r}, J \in \mathbb{R}^{e \times mk} }} &\sum_{t\in\mathcal{T}} \big\| K \mathbf{U}_t^{(k)} +Q_s \mathbf{W}_t^{(s)} + H\alpha_t^{(s)} + H\beta_t\big\|_2 
    \\ \text{s.t.} \quad& HCE = I_e,~~HG^*-J = K, \numberthis\label{expanded2}\\\quad& Q_s = H\cdot [CE~ CAE ~ \dots ~ CA^{s-1}E].
\end{align*}

We assume that Assumptions~\ref{as:fullcolumn1}-\ref{as:fullcolumn} hold throughout the theoretical analysis that follows.

\section{Theoretical Analysis and Main Result}\label{sec:analysis}

In this section, we present the finite-sample analysis of the constrained $\ell_2$-norm estimator \eqref{l2est}. We first establish lower bounds on the control input energy and the interaction terms involving adversarial injections, leading to a universal lower bound for the underlying objective function. Using these preliminary bounds, we prove our main result on the estimation error of the Markov parameter matrix up to a linear transformation.

\subsection{Preliminary Bounds and Objective Decomposition}

To isolate the time steps during which the input trajectory is free from recent historical attack effects, we define the time set
$\mathcal{T}_0 = \{t \in \mathcal{T}: \xi_{t-i} = 0 \text{ for } i=2,\dots,s\}.$

We begin by establishing a high-probability lower bound on the uncorrupted input sequence aggregated over $\mathcal{T}_0$.

\begin{lemma}\label{K_lowerbound}
Let the control inputs $\{u_t\}_{t=0}^{T+k-2}$ be chosen independent with $u_t \sim \mathcal{N}(0, \sigma^2 I_m)$. Consider a fixed matrix $K \in \mathbb{R}^{e \times mk}$. Then, given $\delta \in (0,1]$, when 
$T = \Omega\big( \frac{k}{(1-p)^{s-1}} \log\left(\frac{1}{\delta}\right) \big),$
the inequality
\begin{align}\label{KU_T}
    \sum_{t\in\mathcal{T}_0} \|K \mathbf{U}_t^{(k)}\|_2 \ge \frac{(1-p)^{s-1}\sigma}{4} \sqrt{\frac{2}{\pi}} \|K\|_F T
\end{align}
holds with probability at least $1-\delta$.
\end{lemma}

\medskip
\begin{remark}[Generalization of Control inputs]
Our formulation accommodates a broader class of control input distributions beyond standard Gaussian signals. Following \citet{kim2026sharp}, the requirements on the input vectors can be relaxed to satisfy only sub-Gaussianity along with an expected persistence of excitation condition. Consequently, the control inputs need not be strictly Gaussian, making the proposed defense framework applicable to a wider range of practical excitation signals.
\end{remark}

Next, we establish a lower bound on the interaction term
\begin{align}\label{qsstar}
\sum_{t\in\mathcal{T}}\| K \mathbf{U}_t^{(k)} +Q_s \mathbf{W}_t^{(s)} \|_2 - \|Q_s^* \mathbf{W}_t^{(s)}\|_2,
\end{align} where $Q_s^* :=[I_e ~0~\dots ~0]\in \mathbb{R}^{e\times es}$. 
For the purpose of analysis, select the minimum-spectral-norm matrix $H^*$ satisfying $H^* \cdot [CE~ CAE ~ \dots ~ CA^{s-1}E] = Q_s^*$, whose existence is guaranteed by Lemma~\ref{assumptionequivalence}. The spectral norm of $H^*$ satisfies the following geometric bound.

\begin{lemma}\label{minnormbound}
The  minimum spectral-norm solution $H^*$ satisfying $H^*\cdot [CE~ CAE ~ \dots ~ CA^{s-1}E] = [I_e ~0~\dots ~0]$ 
obeys the geometric bound
$$\|H^*\|_2 \le \frac{1}{\sigma_{\min}(CE) \sin\theta},$$
where $\sigma_{\min}(CE) > 0$ is the smallest singular value of $CE$, and $\theta \in (0, \pi/2]$ is the minimal principal angle between  $\operatorname{Im}(CE)$ and  $\operatorname{Im}([CAE ~ \dots ~ CA^{s-1}E])$.
\end{lemma}

Noting that $ HCE w_{t-1} = w_{t-1}$ and $Q_s^* \mathbf{W}_t^{(s)} = w_{t-1}$, the objective expression \eqref{qsstar} can be decomposed based on the attack indicator $\xi_{t-1}$ as
\begin{subequations}\label{KQfirstsecond}
\begin{align}
  &  \sum_{\substack{t\in\mathcal{T} \\\xi_{t-1}=0}}\| K \mathbf{U}_t^{(k)} +\tilde Q_s \mathbf{\tilde W}_t^{(s)}\|_2 \label{KQfirst}\\&  +  \sum_{\substack{t\in\mathcal{T}\\\xi_{t-1}=1}}\| K \mathbf{U}_t^{(k)} +w_{t-1}+ \tilde Q_s \mathbf{\tilde W}_t^{(s)} \|_2 - \|w_{t-1}\|_2,\label{KQsecond}
  \end{align}
\end{subequations}
where $\tilde Q_s := [HCAE~\dots ~ HCA^{s-1}E]$ and $\mathbf{\tilde W}_t^{(s)}:= [w_{t-2}^T~ \dots ~ w_{t-s}^T]^T$. Applying the reverse triangle inequality to each term in \eqref{KQsecond} yields a unified lower bound on \eqref{KQfirstsecond} as
\begin{align*}
  f_T(K,\tilde Q_s):=  &\sum_{{t=k-1}}^{T+k-2}\|K \mathbf{U}_t^{(k)}  + \tilde Q_s\mathbf{\tilde W}_t^{(s)}\|_2\\&\hspace{5mm}-2 \min\Bigl\{\|K \mathbf{U}_t^{(k)}  + \tilde Q_s\mathbf{\tilde W}_t^{(s)}\|_2,\|w_{t-1}\|_2\Bigr\}.
\end{align*}
Let $\Psi_T(K) := \inf_{\tilde Q_s} f_T(K, \tilde Q_s)$.


\begin{lemma}\label{lem:psi_lower_bound}
Let the control inputs $\{u_t\}_{t=0}^{T+k-2}$ be chosen independent with $u_t \sim \mathcal{N}(0, \sigma^2 I_m)$. For $K \in \mathbb{R}^{e \times mk}$ and $\tilde Q_s \in \mathbb{R}^{e\times e(s-1)}$, consider a function 
\begin{equation}\label{f_T_K_Q}
\begin{split}
    f_T(K,\tilde Q_s):=  &\sum_{t\in\mathcal{T}}\|K \mathbf{U}_t^{(k)}  + \tilde Q_s\mathbf{\tilde W}_t^{(s)}\|_2\\&\hspace{0mm}-2 \min\Bigl\{\|K \mathbf{U}_t^{(k)}  + \tilde Q_s\mathbf{\tilde W}_t^{(s)}\|_2,\|w_{t-1}\|_2\Bigr\}
    \end{split}
\end{equation}
where $\mathbf{\tilde W}_t^{(s)} := [w_{t-2}^T ~ \dots ~ w_{t-s}^T]^T$. 
Given $h>0$ and  $\delta \in (0, 1]$, when
\begin{equation}\label{eq:T_condition_psi}
\begin{split}
    T = \Omega\biggr( &\frac{k}{(1-p)^{s-1}} \log\left(\frac{1}{\delta}\right) \\&\hspace{4mm}+ \frac{\eta^2 e \log(1/\delta)}{ h^2 \rho^{2s}} \log\left(\frac{1}{\delta} \cdot\frac{\eta^2 e \log(1/\delta)}{  \rho^{2s}  h^2} \right) \biggr),
    \end{split}
\end{equation}
the infimum $\Psi_T(K) := \inf_{\tilde Q_s} f_T(K, \tilde Q_s)$ satisfies
\begin{align}\label{psi_lowerbound}
    \Psi_T(K) \ge \frac{(1-2p)(1-p)^{s-1}\sigma}{4} \sqrt{\frac{2}{\pi}} \|K\|_F T -  \rho^s h T.
\end{align}
 with probability at least $1-\delta$.
\end{lemma}

The above lemma holds for a fixed $K\in\mathbb{R}^{e\times mk}$ with high probability.
Now, we extend this result to construct a universal lower bound on $\Psi_T(K)$ over all points in the Frobenius sphere with high probability. 
\begin{lemma}\label{lem:universallowerbound}
 Let the control inputs $\{u_t\}_{t=0}^{T+k-2}$ be chosen independent with $u_t \sim \mathcal{N}(0, \sigma^2 I_m)$.  Define $c_* := \frac{(1-2p)(1-p)^{s-1}\sigma}{4} \sqrt{\frac{2}{\pi}}$ and 
 $R:= \frac{4\rho^s h}{c_*}$ for $h>0$. Let $\mathcal{S}_R := \{K \in \mathbb{R}^{e \times mk} : \|K\|_F = R\}$ denote the Frobenius sphere of radius $R$. Given $\delta\in(0,1]$, when 
 \begin{align}\label{final_timebound}
    &T=\tilde \Omega\left(\frac{e  m k^2s }{(1-p)^{s-1}}  +\frac{\eta^2 e \cdot  (emks)^2 }{ \rho^{2s}h^2 } \right),
\end{align}
we have
 \begin{align}\label{goal}
    \Psi_T(K) \ge \frac{c_*}{4} RT,\quad \forall K\in \mathcal{S}_R 
    \end{align}
    with probability at least $1-\delta$.
\end{lemma}

\subsection{Main Finite-Sample Estimation Guarantee}

Note that $\Psi_T(K)$ lower-bounds the original objective term \eqref{qsstar}. Let $(\hat H, \hat J)$ be a minimizer of the constrained $\ell_2$-norm estimator \eqref{l2est}, and define the estimation error matrix $\hat K := \hat H G^* - \hat J$. We bound $\|\hat K\|_F$ by establishing matching upper and lower bounds on $\Psi_T(\hat K)$.
We provide our main theorem below.

\medskip
\begin{theorem}\label{thm:main}
Let the control inputs $\{u_t\}_{t=0}^{T+k-2}$ be chosen independent with $u_t \sim \mathcal{N}(0, \sigma^2 I_m)$. Consider a minimizer $(\hat H, \hat J)$ to the constrained $\ell_2$-norm estimator \eqref{l2est}, and define $\hat K := \hat H G^* - \hat J$. Given $\delta \in (0, 1]$, when
\begin{align}\label{final_timebound2}
    T = \tilde{\Omega}\bigg( \frac{e m k^2 s}{(1-p)^{s-1}} + \frac{\eta^2 e\cdot (e m k s)^2}{\rho^{2s} (\|\hat H\|_2+\|H^*\|_2)^2} \bigg),
\end{align}
then  the estimation error satisfies
\begingroup
\begin{align}\label{errorbound}
\|\hat K\|_F = \mathcal{O}\bigg(\Big(\frac{\rho}{1-p}\Big)^s\frac{(1-p)(\|\hat H\|_2 + \|H^*\|_2)\nu}{(1-\rho)(1-2p)} \bigg)
\end{align}
\endgroup%
with probability at least $1-\delta$, where $\nu:= \sqrt{m} \|B\|_2\|C\|_2+\frac{1+\eta {\sqrt{n}}\|C\|_2+ \eta\sqrt{e}\|C\|_2\|E\|_2}{\sigma}$.
\end{theorem}

\begin{proof}
For any scalar $\alpha \ge 1$ and $a, b \ge 0$, we have $\min\{\alpha a, b\} \le \alpha \min\{a, b\}$. Applying this inequality to $f_T$ yields $f_T(\alpha K, \alpha Q) \ge \alpha f_T(K, Q)$ for all $\alpha \ge 1$. Taking the infimum over $Q$ on both sides provides
\begin{equation}\label{alphalower}
\begin{split}
    \Psi_T (\alpha K ) &= \inf_{\tilde Q_s} f_T(\alpha K , \tilde Q_s)= \inf_{ Q} f_T(\alpha K , \alpha  Q)\\& \ge \alpha \inf_Qf_T(K,Q) = \alpha \Psi_T(K)
    \end{split}
\end{equation}
for all $\alpha\ge 1$.
Setting $h := \|\hat H\|_2 +\|H^*\|_2> 0$ and $R := \frac{4\rho^s (\|\hat H\|_2+\|H^*\|_2)}{c_*}$, Lemma \ref{lem:universallowerbound} guarantees that when $T$ satisfies \eqref{final_timebound2},
the universal lower bound \eqref{goal} holds
with probability at least $1-\delta$.

Now, for any matrix $\bar K$ satisfying $\|\bar K\|_F \ge R$, consider the scaling parameter $\alpha := \|\bar K\|_F / R \ge 1$. Since $\|\bar K / \alpha\|_F = R$, sequentially applying \eqref{alphalower} and \eqref{goal} yields
\begin{align}\label{psitkbar}
    \Psi_T(\bar K) = \Psi_T\Big( \alpha \frac{\bar K}{\alpha} \Big) \ge \alpha \Psi_T\Big( \frac{\bar K}{\alpha} \Big) \ge \alpha \frac{c_*}{4}  RT = \frac{c_*}{4}  \|\bar K\|_F T.
\end{align}
with probability at least $1-\delta$.

Now, we evaluate the norm of $\hat K = \hat H G^*-\hat J$ under two exhaustive cases:
\begin{itemize}[leftmargin=0.4cm]
    \item \textbf{Case 1 ($\|\hat K\|_F < R$):} By definition of $R$, we  have
    \begingroup
    \begin{align}\label{case1}
    \|\hat K\|_F < \frac{4\rho^s  (\|\hat H\|_2+\|H^*\|_2)}{c_*} = \mathcal{O}\bigg(\frac{\rho^s(\|\hat H\|_2+\|H^*\|_2)}{(1-2p)(1-p)^{s-1} \sigma}\bigg).
\end{align}
\endgroup%
    \item \textbf{Case 2 ($\|\hat K\|_F \ge R$):} By the equivalence between \eqref{l2est} and \eqref{expanded2}, there exist matrices $\hat Q_s,\hat H, \hat J$ such that $(\hat K, \hat Q_s, \hat H, \hat J)$ is a minimizer of \eqref{expanded2}. Furthermore, under Assumptions \ref{as:fullcolumn1} and \ref{as:fullcolumn}, there exist matrices $H^*, J^*$ such that $(0, Q_s^*,  H^*,  J^*)$ is feasible for \eqref{expanded2} by Lemma \ref{assumptionequivalence}. Applying the triangle inequality, we obtain
    \begin{align*}
      & \sum_{t\in\mathcal{T}}  \big\| \hat K \mathbf{U}_t^{(k)} +\hat Q_s \mathbf{W}_t^{(s)}\big\|_2- \|\hat H\|_2 \bigg(\sum_{t\in\mathcal{T}}\|\alpha_t^{(s)}\|_2 + \|\beta_t\|_2\bigg)\\
        &\le\sum_{t\in\mathcal{T}}  \big\| \hat K \mathbf{U}_t^{(k)} +\hat Q_s \mathbf{W}_t^{(s)} + \hat H\alpha_t^{(s)} + \hat H\beta_t\big\|_2 \\&\le \sum_{t\in\mathcal{T}} \big\|  Q_s^* \mathbf{W}_t^{(s)} +  H^*\alpha_t^{(s)} +  H^*\beta_t\big\|_2 \\&\le \sum_{t\in\mathcal{T}}\| Q_s^* \mathbf{W}_t^{(s)}\|_2 + \| H^*\|_2 \bigg(\sum_{t\in\mathcal{T}}\|\alpha_t^{(s)}\|_2 + \|\beta_t\|_2\bigg),
    \end{align*}
    which implies that \begin{align*}\Psi_T(\hat K)&\le \sum_{t\in\mathcal{T}}  \| \hat K \mathbf{U}_t^{(k)} +\hat Q_s \mathbf{W}_t^{(s)}\|_2-\| Q_s^* \mathbf{W}_t^{(s)}\|_2\\&\le (\|\hat H\|_2 + \|H^*\|_2)\bigg(\sum_{t\in\mathcal{T}}\|\alpha_t^{(s)}\|_2 + \|\beta_t\|_2\bigg).\end{align*} Applying Lemma \ref{alphabetabound} to the right-hand side, while using the lower bound \eqref{psitkbar} evaluated at $\bar K = \hat K$ on the left-hand side, we arrive at
    \begin{align*}
        \frac{c_*}{4} \|\hat K\|_F T \le \Psi_T(\hat K) = \mathcal{O}\big( (\|\hat H\|_2 + \|H^*\|_2) \cdot \rho^s \|C\|_2 M T \big),
    \end{align*}
    which simplifies to
    \begin{align}\label{case2}
    \|\hat K\|_F  = \mathcal{O}\bigg(\frac{\rho^s (\|\hat H\|_2 + \|H^*\|_2)}{(1-2p)(1-p)^{s-1} \sigma}\cdot \|C\|_2 M\bigg)
\end{align}
with  probability at least $1-\delta$. 
\end{itemize}
Combining \eqref{case1} and \eqref{case2} together with the definition of $M$ from Lemma~\ref{alphabetabound} completes the proof.
\end{proof}

\medskip
\begin{remark}
    Theorem~\ref{thm:main} shows that the estimation error $\|\hat K\|_F$ is proportional to $\|\hat H\|_2 + \|H^*\|_2$ and decays exponentially with parameter $s$ whenever $\frac{\rho}{1-p} < 1$, or equivalently $p < 1 - \rho$. Combining this with Assumption~\ref{null}, the general convergence condition is $p < \min\{0.5, \ 1-\rho\}$. Since $s$ appears in the time horizon $T$ in \eqref{final_timebound2}, this implies that as the amount of collected data increases, the relative estimation error effectively decreases.
\end{remark}

\subsection{Relative Error Analysis and Implementation Remarks}
 To eliminate the absolute scale of the matrix $\hat H$, we state the relative estimation error bound as a direct corollary.

\begin{corollary}\label{cor:relative_error}
Under the conditions of Theorem \ref{thm:main}, for $\delta \in (0, 1]$, the relative estimation error satisfies
\begin{align}\label{eq:relative_error_bound}
    \frac{\|\hat K\|_F}{\|\hat H\|_2} = \mathcal{O}\bigg(\Big(\frac{\rho}{1-p}\Big)^s\frac{(1-p)(1+ \frac{1}{\sin\theta})\nu}{(1-\rho)(1-2p) } \bigg),
\end{align}
with probability at least $1-\delta$, where $\nu$ is defined in Theorem \ref{thm:main} and $\theta\in(0,\pi/2]$ denotes the minimal principal angle between $\operatorname{Im}(CE)$ and $\operatorname{Im}([CAE ~ \dots ~ CA^{s-1}E])$.
\end{corollary}

\begin{proof}
From the feasibility of $\hat H$, the constraint $\hat H CE = I_e$ implies $\|\hat H\|_2 \ge \frac{1}{\sigma_{\min}(CE)}$. Combining this with Lemma~\ref{minnormbound} leads to $\frac{\|H^*\|_2}{\|\hat H\|_2} \le 
\frac{1 / (\sigma_{\min}(CE)\sin(\theta))}{1 / \sigma_{\min}(CE)} =
\frac{1}{\sin(\theta)}.$
Dividing  \eqref{errorbound} by $\|\hat H\|_2$ and applying this inequality  yields \eqref{eq:relative_error_bound}.
\end{proof}

\medskip
\begin{remark}[Attack Detection]\label{rem:attackdetection}
Equation \eqref{eq:relative_error_bound} provides the metric to identify attacked time steps. Given a pair $(\hat H, \hat J)$,  the estimation residual $\hat H y_t - \hat J \mathbf{U}_t^{(k)}$ expands as $\hat K \mathbf{U}_t^{(k)} + \hat H v_t^{(s)} + \hat H \alpha_t^{(s)} + \hat H \beta_t$. 
Under accurate estimation where $\|\hat K\|_F/\|\hat H\|_2$ is small, all terms except $\hat H v_t^{(s)}$ are reasonably bounded relative to $\|\hat H\|_2$. Thus, the scale of the relative estimation residual is determined by $\hat Hv_t^{(s)}$, which directly captures the immediate prior attack via $\hat H C E w_{t-1} = w_{t-1}$ enforced by the constraint $HCE = I_e$.
Noting from \eqref{relationship} that the approximation $y_t \approx G^* \mathbf{U}_t^{(k)}$ improves as past attacks $w_{t-1}, w_{t-2}, \dots, w_{t-s}$ vanish, this improvement translates to consecutively small relative estimation residuals across the window $\{t, ~t-1, \dots, t-s+1\}$. A time step $t$ satisfying this condition yields an observation $y_t$  classified as clean and thus retained for estimating $G^*$, ensuring robust attack detection and identification.
\end{remark}

\begin{remark}[Preventing Degeneracy]\label{rem:filter_scaling}
If $\|\hat H\|_2$ or $\|H^*\|_2$ is too small, the time bound $T$ in \eqref{final_timebound2} increases significantly. This issue can be readily mitigated by enforcing a lower-bound constraint $\|H\|_2 \ge 1$ in the optimization problem \eqref{l2est}, which leaves the relative estimation error invariant. 
The only exception occurs when $CE$ is square and invertible ($s=1$), where $H = (CE)^{-1}$ is unique and imposing $\|H\|_2 \ge 1$ could induce infeasibility if $\|(CE)^{-1}\|_2 < 1$. In other cases where $H$ is non-unique, the set of valid left-inverses is unbounded, guaranteeing the existence of $H$ satisfying $\|H\|_2 \ge 1$. Imposing this lower-bound constraint prevents $\|\hat H\|_2 + \|H^*\|_2$ from approaching zero, effectively treating the $\|\hat H\|_2 + \|H^*\|_2$ factor in \eqref{final_timebound2} as a constant bounded below by $2$ to ensure a stable time bound.
\end{remark}

\section{Numerical Experiments}\label{sec:exp}

\textit{Experiment 1}: We 
use $n=100$, $m=9$, $e=1$, and $k=10$. We generate the matrix $A$ by selecting all entries from Uniform$[-1,1]$ and rescaling it to satisfy $\rho= 0.5$ (see Assumption \ref{stability}). 
{We consider a networked system of $10$ agents, where each agent $i \in \{1,  \dots, 10\}$ observes the state entries $10i -9, \dots, 10i$, providing an $10$-dimensional partial observation ($r=10$) of the global state $x_t \in \mathbb{R}^{100}$.}
The initial state is set to a vector of $1000$s, the control inputs at each time is designed to follow $\mathcal{N}(0, 100I_{m})$, and the attack probability is set to $p\in\{0.1, 0.3\}$, all satisfying $p<\min\{0.5, 1-\rho\}=0.5$. The attack $w_t$ is sub-Gaussian with variance $25$ and a mean with either $300$ or $1000$, depending on the sign of the first coordinate of $x_t$, making the attack sequence correlated (non-i.i.d.) over time.

We report the relative estimation errors $\frac{\|\hat H G^*-\hat J \|_F}{\|\hat H\|_2}$ across agents $1, \dots, 5$ for $T=1,\dots, 2500$, where the solid line represents the mean and the shaded area indicates $\pm 2$ standard error. We compare our estimator with the least-squares estimator, which replaces $\|\cdot \|_2$ in \eqref{l2est} with $\|\cdot \|_2^2$. Fig.~\ref{fig:error} demonstrates that the $\ell_2$-norm estimator outperforms the least-squares method for all $p\in\{0,1,0.3\}$. Moreover, the figure shows a slower error decay rate as $p$ increases. This validates Theorem~\ref{thm:main}, which establishes that the required time horizon~\eqref{final_timebound2} increases with both $p$ and $s$. Since the error bound~\eqref{errorbound} increases with $p$, reaching a target error level under higher $p$ requires increasing $s$, which in turn necessitates a longer time horizon.

\begin{figure}[t]
     \centering
     \begin{subfigure}[b]{0.235\textwidth}
         \centering
    \includegraphics[width=\textwidth]{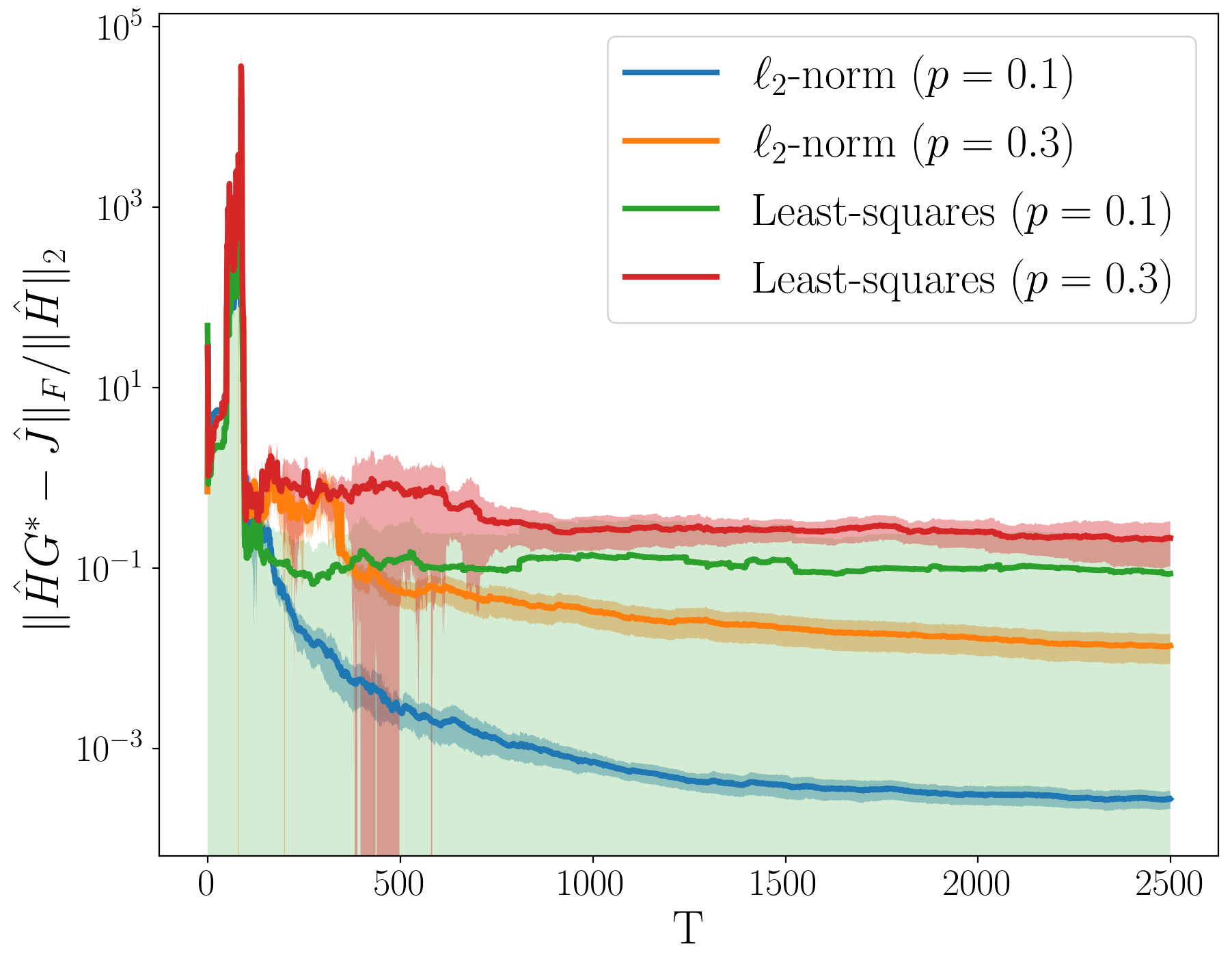}
         \caption{Relative estimation error}
         \label{fig:error}
     \end{subfigure}
     \begin{subfigure}[b]{0.235\textwidth}
         \centering
    \includegraphics[width=\textwidth]{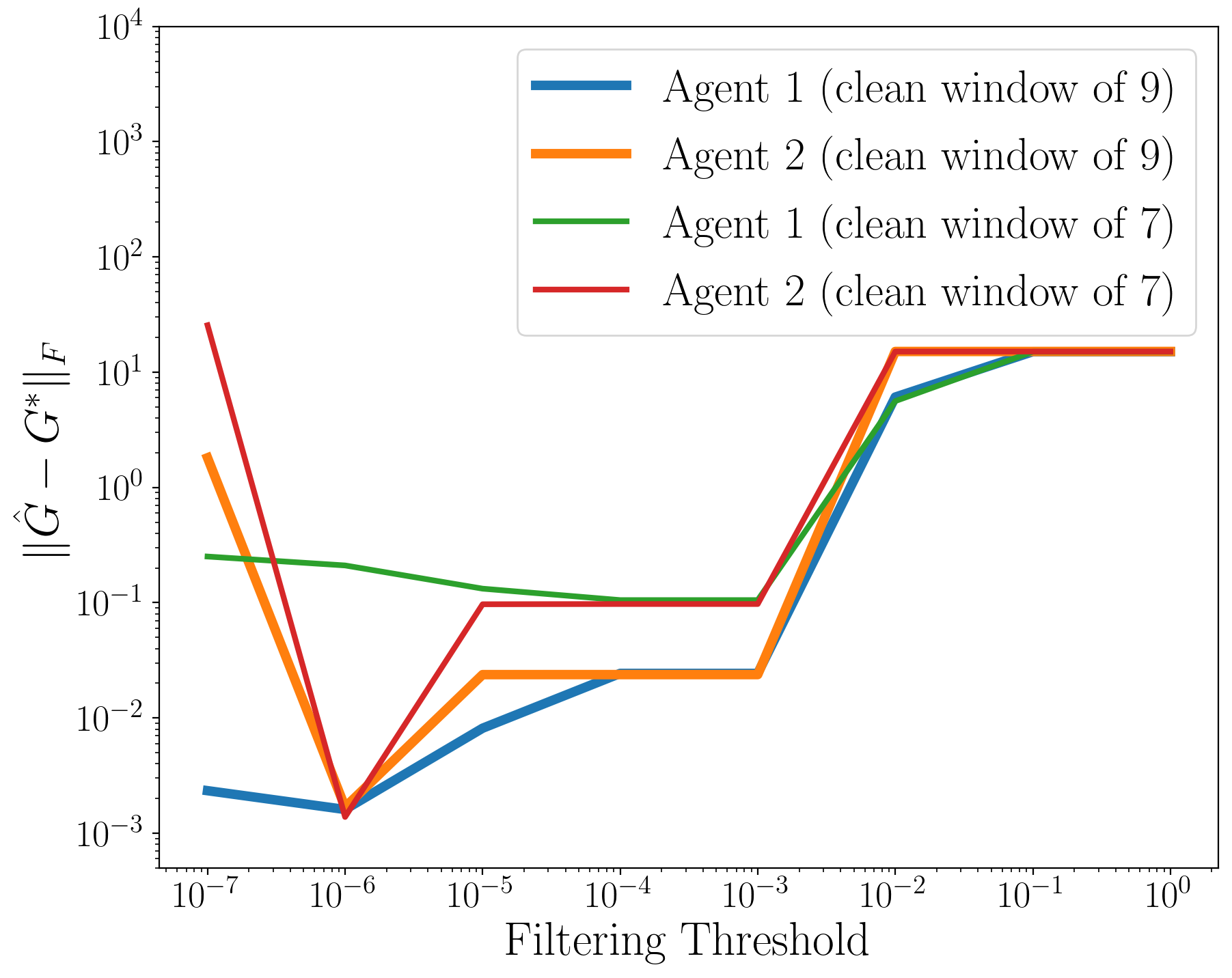}
         \caption{Estimation error for $G^*$}
         \label{fig:filtered}
     \end{subfigure}
        \caption{(a) Relative Estimation error  $\|\hat H G^*-\hat J\|_F / \|\hat H\|_2$ for the Markov parameter matrix $G^*$ of order $k=10$: $\ell_2$-norm estimator vs. least-squares. Solid line = mean, Shaded area = $\pm2$ standard error. \\(b) Estimation error for the same $G^*$: least-squares on the filtered dataset across different thresholds.    
        }
        \label{ex1}
\end{figure}

\textit{Experiment 2}: Following Remark \ref{rem:attackdetection}, we use the estimates $(\hat H, \hat J)$ derived from Experiment 1 (for $p=0.1$) at $T=2500$ to filter out attacked data based on the  relative estimation residuals $\frac{\|\hat H y_t - \hat J \mathbf{U}_t^{(k)}\|_F}{\|\hat H\|_2}$, constructing a filtered time index set $\mathcal{T}_f$  
under a clean-window condition of length $l$: an observation at time $t$ is included in $\mathcal{T}_f$ when the relative estimation residuals across the window $\{t, ~t-1,\dots, t-l+1\}$ consecutively fall below a filtering threshold. 
 Subsequently, we estimate $G^*$ via the nominal least-squares method $\min_{G} \sum_{t\in \mathcal{T}_f}\|y_t - G\mathbf{U}_t^{(k)}\|_2^2$. 
 Note that a higher threshold or a shorter window length removes fewer data points,  approaching the full dataset baseline analyzed in prior works \citep{sarkar2021finite, oymak2022revisit} under independent, zero-mean noise.

Fig.~\ref{fig:filtered} demonstrates the efficacy of our filtering approach in terms of the estimation error $\|\hat G - G^*\|_F$ for agents 1 and 2, where $\hat G$ denotes a minimizer obtained from each estimation method. 
Setting the filtering threshold to $10^{-6}$ and requiring a clean window of length $k-1=9$ results in small estimation errors across both agents, whereas higher filtering thresholds or shorter window length of $7$ yields inaccurate estimates. This behavior is expected, since retaining exclusively attack-free data is necessary to guarantee identification accuracy.
Determining the filtering threshold requires balancing dataset purity and sample efficiency, since overly aggressive filtering may discard useful observations. Thus, optimizing this trade-off remains an important direction for future work.

\section{Conclusion}\label{sec:con}
In this paper, we study the problem of robust system identification for partially observed linear dynamical systems subject to adversarial attacks with an attack probability of $p < 0.5$. By designing a defense filter satisfying full column rank and subspace separability conditions, we introduce a constrained $\ell_2$-norm optimization framework and establish rigorous finite-sample error bounds for recovering the Markov parameter matrix up to a linear transformation. This advances existing approaches that rely on diminishing attack probability as the target estimation error decreases. Furthermore,  our experiments demonstrate that relative estimation residuals provide a reliable diagnostic signal for separating normal and corrupted  time instances, as confirmed by evaluating the standard least-squares method on the filtered dataset. 
To our knowledge, this work is the first to establish a viable defense strategy for partially observed systems under severe adversarial conditions, offering a strong potential for application to multi-agent networks.

\bibliography{ifacconf}             

\appendix
\section{Proofs in Section \ref{sec:probform}}
\subsection{Proof of Lemma \ref{alphabetabound}} 

\begin{proof}
Note that $\|\|x_0\|_2 \|_{\psi_2} \leq \eta{\sqrt{n}}, \|\|w_t\|_2\|_{\psi_2} \leq \eta{\sqrt{e}}, $ and $\|\|u_t\|_2\|_{\psi_2} \leq \sigma\sqrt{m}$  follow from Assumption \ref{maxnorm} and the Gaussianity of $u_t$. 
In turn, using the system dynamics \eqref{sysd} yields
\begin{align*}
  &  \|\|x_t\|_2\|_{\psi_2} = \Bigr\|\Bigr\|A^t x_0 + \sum_{i=0}^{t-1} (A^{t-1-i}Bu_i+A^{t-1-i}E w_{i})\Bigr\|_2 \Bigr\|_{\psi_2} \\& \leq \|A^t\|_2\cdot\|\|x_0\|_2\|_{\psi_2} \\&\hspace{5mm}+ \sum_{i=0}^{t-1} \|A^{t-1-i}\|_2 \bigr(\|B\|_2\cdot \|\|u_i\|_2\|_{\psi_2} +\|E\|_2 \|\|w_i\|_2\|_{\psi_2}\bigr) \\&\le \frac{\psi}{1-\rho} \cdot \left[\eta {\sqrt{n}}+ \eta\sqrt{e}\|E\|_2 + \sigma \sqrt{m} \|B\|_2\right]=\mathcal{O}(M)
    \numberthis\label{xtsums}
\end{align*}
due to the triangle inequality and geometric sum with Assumption \ref{stability}.  We can sum up $\|x_t\|_2$ over $t=0, \dots, T-1$ to yield
\begin{align*}
        &\sum_{t=0}^{T-1} \|x_t\|_2 
        \le\sum_{i=0}^{\infty} \|A^i\|_2   \Bigr[\|x_0\|_{2} + \sum_{t=0}^{T-2} (\|Ew_t\|_{2} +  \|Bu_t\|_{2})\Bigr]\\&\hspace{2mm}\leq\frac{\psi}{1-\rho}\Bigr[\|x_0\|_{2} + \sum_{t=0}^{T-2} (\|E\|_2\|w_t\|_{2} +  \|B\|_2\|u_t\|_{2})\Bigr]:= X.
    \end{align*}
{Since the filtration $\mathcal{F}_t$ can alternatively be written as $\bm{\sigma}\{x_0, \{(u_\tau, w_\tau)\}_{\tau=0}^{t-1}\}$, sequentially applying the tower rule using $\mathcal{F}_t$ over $T$ steps establishes that $X-\mathbb{E}[X]$ is sub-Gaussian with norm $\mathcal{O}(M\sqrt{T})$.} 
   Then, Hoeffding's inequality yields
    \begin{align*}
        &\mathbb{P}\bigr( X-\mathbb{E}[X] \leq MT  \bigr) \geq 1-\exp(-\Omega(T)).\numberthis\label{xt123}
    \end{align*}
Then, considering $\sum_{t=0}^{T-1} \|x_t\|_2 \le X$ and the expectation bound 
    $\mathbb{E}\bigr[\sum_{t=0}^{T-1} \|x_t\|_2\bigr]=\mathcal{O}(MT)$ from \eqref{xtsums} yields
    \begin{align*}
        \sum_{t=k-1}^{T+k-2} \|\beta_t\|_2= \sum_{t=0}^{T-1}\|CA^{k-1}x_t\|_2 = \mathcal{O}( \rho^{k-1}  \|C\|_2MT )\numberthis\label{betatsum}
    \end{align*}
with probability at least $1-\frac{\delta}{2}$. Similarly, following  the same filtration argument, we have 
\begin{align*}
    \sum_{t=k-1}^{T+k-2}\|\alpha_t^{(s)}\|_2& \le\mathcal{O}\Big( \rho^s\|C\|_2\frac{\sum_{t=0}^{T+k-4}\|E\|_2\|w_t\|_2}{1-\rho}\Big) \\& = \mathcal{O}\Big( \rho^s \|C\|_2\frac{\eta\sqrt{n}\|E\|_2}{1-\rho}T\Big)\numberthis\label{alphatsum}
\end{align*}
with probability at least $1-\frac{\delta}{2}$. We complete the proof by applying the union bound and summing up \eqref{betatsum} and \eqref{alphatsum}.
\end{proof}

\subsection{Proof of Lemma \ref{assumptionequivalence}}
\begin{proof}
For  simplicity, let $L := [CAE ~ \dots ~ CA^{s-1}E]$. 

$(\Longrightarrow)$ Suppose such a matrix $H \in \mathbb{R}^{e \times r}$ exists. 
First, $H CE = I_e$ implies $e = \operatorname{rank}(I_e) = \operatorname{rank}(H CE) \le \operatorname{rank}(CE) \le e$, so $\operatorname{rank}(CE) = e$.
Second, let $y \in \operatorname{Im}(CE) \cap \operatorname{Im}(L )$. Then $y = CE u = L  v$ for some $u \in \mathbb{R}^e$ and $v \in \mathbb{R}^{e(s-1)}$. Multiplying by $H$ yields:
\[
u = I_e u = H CE u = H y = H L  v = 0 \cdot v = 0.
\]
Thus $y =  0$, proving $\operatorname{Im}(CE) \cap \operatorname{Im}(L ) = \{0\}$.

\medskip
$(\Longleftarrow)$ Suppose $\operatorname{rank}(CE) = e$ and $\operatorname{Im}(CE) \cap \operatorname{Im}(L) = \{0\}$. 
Let $V = \{x+y\in\mathbb{R}^r: x\in\operatorname{Im}(CE), y\in \operatorname{Im}(L) \}$. Any vector $z \in V$ can be expressed as $z = CE u + L v$. If 
\[
CE u_1 + L v_1 = CE u_2 + L v_2,
\]
then $CE(u_1 - u_2) = L (v_2 - v_1) \in \operatorname{Im}(CE) \cap \operatorname{Im}(L) = \{0\}$. Hence $CE(u_1 - u_2) = 0$, and since $\operatorname{rank}(CE) = e$, it follows that $u_1 = u_2$.

Therefore, the linear map $H_0 : V \to \mathbb{R}^e$ given by
\[
H_0(CE u + L v) = u
\]
is well-defined because $u$ is uniquely determined for every $z \in V$. Evaluation at $v = 0$ and $u = 0$ respectively gives $H_0 CE = I_e$ and $H_0 L = 0$.

Finally, extend $H_0$ to all of $\mathbb{R}^r$ by choosing a complementary subspace $W$ such that $\mathbb{R}^r = V + W$ with $V \cap W = \{0\}$, and setting $H(w) = 0$ for all $w \in W$. This yields the desired matrix $H\in\mathbb{R}^{e\times r}$.
\end{proof}

\section{Proofs in Section \ref{sec:analysis}}

\subsection{Proof of Lemma \ref{K_lowerbound}}
\begin{proof}
For a fixed realization of $\mathcal{T}_0$, define $f_K(\mathbf{U}) := \sum_{t\in\mathcal{T}_0} \|K \mathbf{U}_t^{(k)}\|_2$ for $\mathbf{U} = [u_0^T ~ u_1^T ~\dots u_{T+k-2}^T]^T \in \mathbb{R}^{m(T+k-1)}$.

First, we compute the Lipschitz constant of $f_K$. For any $\mathbf{U}, \mathbf{U}' \in \mathbb{R}^{m(T+k-1)}$, Cauchy-Schwarz inequality yields
\begin{align*}
    |f_{K}(\mathbf{U}) - f_{K}(\mathbf{U}')| 
    &\le \sum_{t\in\mathcal{T}_0} \|K (\mathbf{U}_t^{(k)} - (\mathbf{U}_t')^{(k)})\|_2 \\
    &\le \|K\|_F \sum_{t\in\mathcal{T}_0} \|\mathbf{U}_t^{(k)} - (\mathbf{U}_t')^{(k)}\|_2 \\
    &\le \|K\|_F \sqrt{|\mathcal{T}_0|} \sqrt{\sum_{t\in\mathcal{T}_0} \|\mathbf{U}_t^{(k)} - (\mathbf{U}_t')^{(k)}\|_2^2} \\
    &\le \|K\|_F \sqrt{|\mathcal{T}_0| k} \cdot \|\mathbf{U} - \mathbf{U}'\|_2.\numberthis\label{KUlipschitz}
\end{align*}
Thus, $f_K$ has a Lipschitz constant of  $\|K\|_F \sqrt{|\mathcal{T}_0| k}$.

Gaussian concentration for Lipschitz functions (see Lemma 1 in \cite{kim2025bridging}) implies that
the sub-Gaussian norm of $\sum_{t\in\mathcal{T}_0} \|K\mathbf{U}_t^{(k)}\|_2-\mathbb{E}[\|K\mathbf{U}_t^{(k)}\|_2]$ is bounded by $\mathcal{O}(\sigma\|K\|_F\sqrt{|\mathcal{T}_0|k})$. Thus, we can write
\begin{align*}
   & \mathbb{P}\biggr( \sum_{t\in\mathcal{T}_0} \|K \mathbf{U}_t^{(k)}\|_2-\mathbb{E}[\|K \mathbf{U}_t^{(k)}\|_2] < -\frac{\sigma }{2}\sqrt{\frac{2}{\pi}}\|K\|_F|\mathcal{T}_0|\biggr)\\&\hspace{55mm}\le \exp\left(-\Omega\biggr(\frac{|\mathcal{T}_0| }{k}\right)\biggr).
\end{align*}
By $\sum_{t\in\mathcal{T}_0}\mathbb{E}[\|K \mathbf{U}_t^{(k)}\|_2]\ge \sigma\sqrt{\frac{2}{\pi}}\|K\|_F |\mathcal{T}_0|$, we arrive at
\begin{align*}
   & \mathbb{P}\biggr( \sum_{t\in\mathcal{T}_0} \|K \mathbf{U}_t^{(k)}\|_2<  \frac{\sigma }{2}\sqrt{\frac{2}{\pi}}\|K\|_F|\mathcal{T}_0|\biggr)\le \exp\left(-\Omega\biggr(\frac{|\mathcal{T}_0| }{k}\right)\biggr).
\end{align*}

We now bound the success probability using a union bound over two failure events:
\begin{enumerate}
    \item $E_1$: The event that $|\mathcal{T}_0| \ge \frac{(1-p)^{s-1}}{2} T$. By Chernoff's bound, $\mathbb{P}(E_1^c) \le \exp\left( -\frac{1}{8}(1-p)^{s-1} T \right)$. Setting this bound to $\le \delta/2$ requires $T = \Omega\big(\frac{1}{(1-p)^{s-1}} \log\left(\frac{1}{\delta}\right)\big)$.
    \item $E_2$: The event that $\sum_{t\in\mathcal{T}_0} \|K \mathbf{U}_t^{(k)}\|_2 \ge \frac{\sigma}{2}\sqrt{\frac{2}{\pi}}\|K\|_F |\mathcal{T}_0|$. Conditioned on $E_1$, we have $|\mathcal{T}_0| \ge \frac{(1-p)^{s-1}}{2} T$, leading to 
   $ \mathbb{P}(E_2^c \mid E_1) \le \exp\bigr( -\Omega\bigr(\frac{(1-p)^{s-1} T}{ k} \bigr)\bigr).$
    Setting this bound to $\le \delta/2$ requires $T =\Omega\big(\frac{ k}{(1-p)^{s-1}} \log\left(\frac{1}{\delta}\right)\big)$.
\end{enumerate}

Since $k\ge 1$, by taking $T =\Omega\big(\frac{ k}{(1-p)^{s-1}} \log\left(\frac{1}{\delta}\right)\big)$, the union bound guarantees that $E_1 \cap E_2$ holds with probability at least $1 - \mathbb{P}(E_1^c) - \mathbb{P}(E_2^c \mid E_1) \ge 1 - \delta$.

Finally, on the event $E_1 \cap E_2$, we substitute the lower bound for $|\mathcal{T}_0|$ in $E_1$ directly into the lower bound for $\sum_{t\in\mathcal{T}_0} \|K \mathbf{U}_t^{(k)}\|_2$ in $E_2$.
This completes the proof.
\end{proof}

\subsection{Proof of Lemma \ref{minnormbound}}
\begin{proof}
For simplicity, let $L = \begin{bmatrix} CAE & \dots & CA^{s-1}E \end{bmatrix}$. Also, let $P_L^\perp = I_r - L L^+$, where $L^+$ is the Moore-Penrose pseudo-inverse of $L$. The constraint $H^* L = 0$ parameterizes $H^* = X P_L^\perp$, which reduces $H^* CE = I_e$ to $X (P_L^\perp CE) = I_e$. Under Assumptions~\ref{as:fullcolumn1} and~\ref{as:fullcolumn}, $P_L^\perp CE$ has full column rank $e$, yielding the unique minimum spectral-norm solution $H^* = (P_L^\perp CE)^+$ with $\|H^*\|_2 = \frac{1}{\sigma_{\min}(P_L^\perp CE)}$.

For any unit vector $x \in \mathbb{R}^e$, let $y = CE x \in \operatorname{Im}(CE)$. By the definitions of singular values and the minimal principal angle $\theta$, we have
\[
\|P_L^\perp CE x\|_2 = \|P_L^\perp y\|_2 \ge \sin(\theta) \|y\|_2 \ge \sin(\theta) \sigma_{\min}(CE).
\]
Taking the infimum over $\|x\|_2 = 1$ yields $\sigma_{\min}(P_L^\perp CE) \ge \sigma_{\min}(CE) \sin(\theta) > 0$. Taking the reciprocal completes the proof.
\end{proof}

\subsection{Proof of Lemma \ref{lem:psi_lower_bound}}

\begin{proof}
Define the $\sigma$-algebra $\mathcal{U} := \sigma(\{u_t\}_{t=0}^{T+k-2})$ and the growing filtration
\begin{align*}
    \mathcal{G}_{t-1} := \sigma\left( \mathcal{U}, \, \{w_{t-2}, w_{t-3}, \dots, w_0\} \right).
\end{align*}
Note that $\mathbf{U}_t^{(k)} \in \mathcal{G}_{t-1}$ and $\mathbf{\tilde W}_t^{(s)} \in \mathcal{G}_{t-1}$, which implies that $\|K \mathbf{U}_t^{(k)} + \tilde Q_s \mathbf{\tilde W}_t^{(s)}\|_2$ is $\mathcal{G}_{t-1}$-measurable. Moreover, the indicator of new attack satisfies $\mathbb{P}(\xi_{t-1}=1 \mid \mathcal{G}_{t-1}) = p$.

Fix a threshold $B > 0$ and introduce the bounded, $\mathcal{G}_{t-1}$-measurable process $r_t := \min\big\{ \|K \mathbf{U}_t^{(k)} + \tilde Q_s \mathbf{\tilde W}_t^{(s)}\|_2, B \big\}$. Define the high-probability noise-truncation event
\begin{align*}
    \mathcal{E}_B := \left\{ \max_{t\in\mathcal{T}} \|w_{t-1}\|_2 \le B \right\}.
\end{align*}
On event $\mathcal{E}_B$, we have $\min\{\|K \mathbf{U}_t^{(k)} + \tilde Q_s \mathbf{\tilde W}_t^{(s)}\|_2, \|w_{t-1}\|_2\} \le \xi_{t-1} r_t$ since $\xi_{t-1}=0$ implies $w_{t-1}=0$. Substituting this into \eqref{f_T_K_Q} and centering around $p$ yields:
\begin{align*}
  &  f_T(K,\tilde Q_s) \ge
\sum_{t\in\mathcal{T}} \|K \mathbf{U}_t^{(k)}  + \tilde Q_s\mathbf{\tilde W}_t^{(s)}\|_2-2\sum_{t\in\mathcal{T}}\xi_{t-1}r_t\\
    &\ge (1-2p) \sum_{t\in\mathcal{T}} \|K \mathbf{U}_t^{(k)} + \tilde Q_s\mathbf{\tilde W}_t^{(s)}\|_2 - 2\sum_{t\in\mathcal{T}}(\xi_{t-1}-p) r_t.\numberthis\label{1-2plower}
\end{align*}
Let $X_t := 2(\xi_{t-1}-p)r_t$. Then, since $\mathbb{E}[X_t \mid \mathcal{G}_{t-1}] = 0$, $\{X_t\}$ forms a martingale difference sequence with respect to $\{\mathcal{G}_{t-1}\}$.
Since $X_t \in [-2pr_t, \, 2(1-p)r_t]$ conditioned on $\mathcal{G}_{t-1}$, its conditional range is bounded by $2r_t \le 2B$. By Hoeffding's lemma for bounded variables, $\mathbb{E}\left[ e^{\lambda X_t} \mid \mathcal{G}_{t-1} \right] \le \exp\big( \frac{\lambda^2 B^2}{2} \big)$. Sequentially applying the tower rule using $\mathcal{G}_{t-1}$ over $T$ steps yields $\mathbb{E}\left[ \exp\left(\lambda \sum_{t\in\mathcal{T}} X_t\right) \right] \le \exp\big( \frac{\lambda^2 T B^2}{2} \big)$. By Hoeffding's inequality, we have
\begin{align}\label{eq:X_concentration}
    \mathbb{P}\left( \sum_{t\in\mathcal{T}} X_t \ge B \sqrt{2T \log\Big(\frac{3}{\delta}\Big)} \right) \le \frac{\delta}{3}.
\end{align}
Furthermore,
since $\|w_t\|_2$ is sub-Gaussian with parameter $\eta\sqrt{e}$, there exists a constant $c>0$ such that $\mathbb{P}(\|w_t\|_2 > B) \le 2\exp\big( -\frac{c B^2}{ \eta^2 e} \big)$. A union bound over all $T$ steps gives $\mathbb{P}(\mathcal{E}_B^c) \le 2T \exp\left( -\frac{c B^2}{ \eta^2 e} \right)$. Choosing 
\begin{align}\label{eq:B_choice}
    B =\sqrt{\frac{\eta^2 e}{c}  \log\left(\frac{6T}{\delta}\right) }
\end{align}
ensures that $\mathbb{P}(\mathcal{E}_B) \ge 1 - \frac{\delta}{3}$.

Finally, invoking Lemma \ref{K_lowerbound}, to guarantee that inequality \eqref{KU_T} holds with probability at least $1-\frac{\delta}{3}$, it suffices that $T = \Omega\left( \frac{k}{(1-p)^{s-1}} \log\left(\frac{3}{\delta}\right) \right)$.

Substituting \eqref{KU_T} and \eqref{eq:X_concentration} into \eqref{1-2plower}, we obtain
\begin{equation}\label{ft_lowerbound}
\begin{split}
  &  f_T(K, \tilde Q_s) 
    \\
    &\ge \frac{(1-2p)(1-p)^{s-1}\sigma}{4} \sqrt{\frac{2}{\pi}} \|K\|_F T - B \sqrt{2T \log\Big(\frac{3}{\delta}\Big)}.
    \end{split}
\end{equation}

We bound the last term as $B \sqrt{2T \log(\frac{3}{\delta})}\le  \rho^s h T$, which translates to $T \ge \frac{2 B^2 \log(3/\delta)}{  \rho^{2s} h^2}$. Substituting \eqref{eq:B_choice} for $B$ yields the sufficient condition on $T$ in \eqref{eq:T_condition_psi}. By applying a union bound across all events, the lower bound \eqref{ft_lowerbound} holds under \eqref{eq:T_condition_psi} with probability at least $1-\delta$.

Taking the infimum over $\tilde Q_s$ on both sides of \eqref{ft_lowerbound} yields \eqref{psi_lowerbound}, concluding the proof.
\end{proof}

\subsection{Proof of Lemma \ref{lem:universallowerbound}}
 
\begin{proof}
We establish the uniform lower bound via an $\epsilon$-net discretization argument over $\mathcal{S}_R$.
To this end, recall that $f_T(K, \tilde Q_s)$ can be equivalently represented using the identity $a - 2\min\{a, b\} = |a - b| - b$:
\begingroup
\fontsize{9.5pt}{11.5pt}\selectfont
\begin{align*}
    &f_T(K, \tilde Q_s) = \sum_{t\in\mathcal{T}} \Big| \|K \mathbf{U}_t^{(k)} + \tilde Q_s \mathbf{\tilde W}_t^{(s)}\|_2 - \|w_{t-1}\|_2 \Big| - \|w_{t-1}\|_2.
\end{align*}
\endgroup%
For any $K, \tilde K \in \mathcal{S}_R$ and any matrix $\tilde Q_s$, applying the reverse triangle inequality yields
\begin{align*}
    &|f_T(K, \tilde Q_s) - f_T(\tilde K, \tilde Q_s)| 
    \\&\le \sum_{t\in\mathcal{T}} \left| \|K \mathbf{U}_t^{(k)} + \tilde Q_s \mathbf{\tilde W}_t^{(s)}\|_2 - \|\tilde K \mathbf{U}_t^{(k)} + \tilde Q_s \mathbf{\tilde W}_t^{(s)}\|_2 \right| \\
    &\le \sum_{t\in\mathcal{T}} \|(K - \tilde K) \mathbf{U}_t^{(k)}\|_2 \le \|K - \tilde K\|_F \sum_{t\in\mathcal{T}} \|\mathbf{U}_t^{(k)}\|_2.
\end{align*}
Rearranging and taking the infimum over $\tilde Q_s$ on both sides gives
\begin{align}\label{eq:psi_lipschitz}
    |\Psi_T(K) - \Psi_T(\tilde K)| \le \|K - \tilde K\|_F \sum_{t\in\mathcal{T}} \|\mathbf{U}_t^{(k)}\|_2.
\end{align}
Similar to the relationship \eqref{KUlipschitz}, the term $\sum_{t\in\mathcal{T}} \|\mathbf{U}_t^{(k)}\|_2$ is $\sqrt{Tk}$-Lipschitz. 
Since we have $\mathbb{E}[\|\mathbf{U}_t^{(k)}\|_2] \le \sigma \sqrt{mk}$ from $\mathbf{U}_t^{(k)} \sim \mathcal{N}(0, \sigma^2 I_{mk})$, 
Gaussian concentration for Lipschitz functions again implies that 
\begin{align*}
    \mathbb{P}\left( \sum_{t\in\mathcal{T}} \|\mathbf{U}_t^{(k)}\|_2 \le 2  \sigma \sqrt{mk}\cdot T \right) \ge 1 -  \exp(-\Omega(mT)).
\end{align*}
Thus, when $T=\Omega\bigr( \frac{1}{m}\log\left(\frac{2}{\delta}\right)\bigr)$, 
the inequality \eqref{eq:psi_lipschitz} extends to 
\begin{align}\label{Kinterbound}
    |\Psi_T(K) - \Psi_T(\tilde K)| \le 2 \sigma \sqrt{mk} \cdot T  \|K - \tilde K\|_F
\end{align}
with probability at least $1-\frac{\delta}{2}$.
To bound the right-hand side by $\frac{c_*}{2} R T$, we set the discretization grid resolution to
$\epsilon := \frac{c_* R}{4 \sigma \sqrt{mk}}.$
The covering number $N_\epsilon$ needed to cover $\mathcal{S}_R$ with $\epsilon$-balls in Frobenius norm satisfies
\begin{align*}
    N_\epsilon \le \left( 1 + \frac{2R}{\epsilon} \right)^{e m k} =\left( 1 +  \frac{16\sqrt{2\pi}\cdot \sqrt{mk}}{(1-2p)(1-p)^{s-1}} \right)^{e m k}.
\end{align*}
For each fixed grid point $\tilde K $ in the $\epsilon$-net, applying Lemma \ref{lem:psi_lower_bound} with failure probability $\frac{\delta}{2N_\epsilon}$ ensures that $\Psi_T(\tilde K) \ge c_* R T - \rho^s h T$ with probability at least $1 - \frac{\delta}{2N_\epsilon}$.

By a union bound over all $N_\epsilon$ grid points in the $\epsilon$-net, the pointwise lower bound holds simultaneously across the entire net with probability at least $1 - \frac{\delta}{2}$. Combining this with the  event in \eqref{Kinterbound} (which holds with probability at least $1 - \frac{\delta}{2}$), for any $K \in \mathcal{S}_R$, choosing a net point $\tilde K$ such that $\|K - \tilde K\|_F \le \epsilon$ yields
\begin{align*}
    \Psi_T(K) &\ge \Psi_T(\tilde K) - 2 T \sigma \sqrt{mk} \, \epsilon \\
    &\ge \left( c_* R T - \rho^s h T \right) - \frac{c_*}{2} R T = \frac{c_*}{4} R T
\end{align*}
with probability at least $1 - \delta$, where the last equality is from the definition of $R$.

To guarantee the union bound, we replace  $\delta $ in \eqref{eq:T_condition_psi} with  $\frac{\delta}{2N_\epsilon}$. Expanding the logarithm yields $\log(\frac{2N_\epsilon}{\delta}) = \mathcal{O}\big( e m k \log\big( \frac{mk}{(1-2p)(1-p)^{s-1}} \big) + \log(\frac{1}{\delta}) \big)=\mathcal{\tilde O}(emks)$, which leads to
the final time bound in \eqref{final_timebound}, hiding logarithmic factors. Note that  $T=\Omega\bigr( \frac{1}{m}\log\left(\frac{2}{\delta}\right)\bigr)$ required for \eqref{Kinterbound} is dominated by \eqref{final_timebound}. This completes the proof.
\end{proof}

\end{document}